\documentclass[12pt]{elsarticle}
\usepackage{amssymb}
\usepackage{amsfonts}
\usepackage{amsmath}
\usepackage{fullpage}
\usepackage{mathrsfs}
\usepackage[toc,page]{appendix}
\usepackage{eurosym}

\makeatletter
\def\ps@pprintTitle{  \let\@oddhead\@empty
  \let\@evenhead\@empty
  \def\@oddfoot{\reset@font\hfil\thepage\hfil}
  \let\@evenfoot\@oddfoot
}
\makeatother
\newtheorem{theorem}{Theorem}

\newtheorem{condition}[theorem]{Condition}

\newtheorem{corollary}[theorem]{Corollary}

\newtheorem{definition}[theorem]{Definition}

\newtheorem{lemma}[theorem]{Lemma}

\newtheorem{remark}[theorem]{Remark}

\newenvironment{proof}[1][Proof]{\noindent\textbf{#1.} }{\ \rule{0.5em}{0.5em}}
\input{tcilatex}
\begin{document}

\begin{frontmatter}

\title{Dimension estimates for nonlinear nonautonomous systems in arbitrary normed spaces}

\author[label1]{John Ioannis Stavroulakis}

\address[label1]{School of Mathematics, Georgia Institute of Technology, Atlanta, GA
30332, USA}

\begin{abstract}
We calculate explicit estimates for the dimension of trajectories satisfying a certain growth bound. We generalize the classic results of Kurzweil by considering nonlinear nonautonomous and uniformly compact dynamical systems on normed spaces over arbitrary fields. We furthermore refine the results in the case of delay equations, greatly simplifying the relevant growth bounds.
\end{abstract}

\begin{keyword}
Attractors, Dimension, Nonarchimedean, Delay equation

\noindent
{\bf AMS subject classification:} 
	37L30; 34K12, 34K30, 37C35,	37C45, 37C60 
\end{keyword}

\end{frontmatter}

\section{\protect\bigskip \protect\bigskip Introduction}

The pioneering dimension estimates of \textquotedblleft
attractors\textquotedblright\ \ for the general case of a compact dynamical
system in an infinite dimensional space were those of Kurzweil \cite{kurz1},
Mallet-Paret \cite{mpdim}, and Yorke \cite{yorkedim}, all three studying the
problem from the perspective of delay equations. Yorke constructed examples
showing that, without smoothness restrictions, it is impossible to bound the
dimension of the attractor of general systems, as there exist continuous
functionals $G$ so that the attractor of 
\begin{equation*}
x^{\prime }(t)=G(x_{t}),
\end{equation*}%
is infinite dimensional. On the other hand, for the linear nonautonomous
case, Kurzweil calculated explicit bounds on the dimension of solutions of
linear maps in Banach spaces satisfying a certain growth estimate, under
uniform compactness. Mallet-Paret proved finiteness of Hausdorff dimension
of invariant sets for autonomous $\mathscr{C}^{1}$ maps in Hilbert spaces
possessing compact derivative. In other words, dimension estimates are
possible only in the linear case, or for systems which are approximately
linear (smooth). However, given linearity, dimension estimates are possible
both for autonomous and nonautonomous systems, using somewhat similar
techniques.

Mallet-Paret's results were subsequently generalized by Ma\~{n}\'{e} \cite%
{mane} and Douady and Oesterl\'{e} \cite{do}. The methods of \cite{do}
inspired a series of successive improvements and generalizations of
dimension estimates for invariant sets of autonomous systems, see e.g. \cite%
{cft,eden,eft,gt,t}. The current state-of-the-art for the autonomous case is
contained in \cite{th,tha}.

Investigations of the nonautonomous case (\cite{clr, cr,cv,haleinfinite,mik,
robin, t}) typically rely on the technique of Mallet-Paret \cite{mpdim} and
Ma\~{n}\'{e} \cite{mane}, reducing a nonautonomous problem to an autonomous
question of invariance, in a manner not dissimilar to the \textit{natural
extension} (cf. \cite[Appendix]{tha}) or the \textit{Howland lift }(on
which, see e.g. \cite{grs, how, lms} and the discussion in \cite[Chapter 3]%
{cl}). Quite similar methods of lifting were recently developed in \cite{bv}%
. Thus, the nonautonomous results so far demand rather restrictive
conditions on the global trajectories, and calculate bounds in a
nonconstructive, nonexplicit manner.

The present paper aims to complete the investigation of \cite{kurz1, kurz},
estimating the angular distance between unstable directions, thereby
providing a basis for dimension estimates in the nonlinear case. We extend
the results to a sequence of maps acting on a sequence of different vector
spaces, which are allowed to be over any field. So for example, our results
may be applied to dynamical systems on a vector space over the $p$-adics. We
furthermore precisely calculate the intimate relationship between the
valuation of the underlying field, the rate of growth, the dimension, and
\textquotedblleft eigenvalues, singular values\textquotedblright\ of the
solution/shift maps.

The one restriction we impose is uniform bounds on these \textquotedblleft
eigenvalues\textquotedblright\ (for numerous versions of these growth
estimates see e.g. \cite{b,by,kurz1,p, qtz}), which is a natural assumption
equivalent to \textit{integral boundedness} (also sometimes referred to as%
\textit{\ translation boundedness}) for various systems. As an illustration,
we indicate how to easily refine the classical estimates of \cite{kurz} for
delay equations with integrally bounded parameters, as a matter of simple
rescaling.

This paper is structured as follows. In the following Section, we extend the
calculations of Kurzweil \cite{kurz1} to spaces over arbitrary fields,
furthermore bounding the distance between unstable directions. Subsequently,
we simplify and refine the growth estimates for delay equations.

\section{\protect\bigskip \protect\bigskip Dimension Estimates for
Nonautonomous Systems}

\subsection{Preliminaries and standing assumptions}

Let us begin by some methodological remarks. Attempting to tackle the
nonautonomous case similarly to Thieullen's work \cite{th,tha}, one faces
the problem of estimating the \textit{covering number}, similarly to \cite[%
p. 83]{th}. This requires bounding the norm of the linearization on a finite
collection of unstable directions, as well as a uniform bound for the
truncation/complement of this finite subspace. The former step is\ -at least
today- wellknown, utilizing the notion of a \textit{determinant} \cite%
{b,by,ghrs,gq,kurz1,qtz}. The latter, however, is considerably more
difficult, if not impossible to achieve without assuming lower bounds on the
determinant, or losing crucial generality in the freedom of choice of the
finite collection of unstable directions in question. \vskip1pt For this
reason, we here closely follow the method of \cite{kurz1} instead of \cite%
{th,tha}, firstly considering only finite dimensional unstable subspaces. In
a second step, we show that, given the existence of a finite family of
unstable directions, all other unstable directions will be in arbitrarily
small neighborhoods of it. \vskip1pt The Lemmata of this section are
straightforward extensions of \cite{kurz71, kurz72, kurz, kurz1, kurz73} to
the nonarchimedean case. We consider (infinite dimensional) normed vector
spaces over arbitrary fields. On nonarchimedean fields and vector spaces, we
refer the reader to \cite{dr,monna,schn,vr}. More recent variations on the
theme of determinants and volumes can be found in \cite{b,by,ghrs,gq,qtz}
for archimedean fields and in \cite{dis, be, bgm,loop} for the
nonarchimedean case. \vskip1pt The normed spaces are not necessarily
complete, and the case of the trivial valuation is not considered. The
symbol $\mathbb{K}$ will denote the relevant field. We study a given
sequence of normed linear spaces $Y_{n},n\in 
\mathbb{Z}
$ over a fixed field $\mathbb{K}$ and of linear operators $%
T_{n}:Y_{n}\rightarrow Y_{n+1},n\in 
\mathbb{Z}
$. We write $T^{n}:=T_{n}\circ ...\circ T_{0}$. We shall assume they
uniformly (wrt discrete time $n$) satisfy the following condition.

\begin{condition}
\label{Condition4.55}There exists a sequence of subspaces $X^{i}$ of $Y_{n}$
(which may be different for each $Y_{n}$), such that 
\begin{eqnarray*}
Y_{n} &=&X^{0}\supset ...\supset X^{i}\supset X^{i+1}\supset ... \\
\func{codim}X^{i} &\leq &k_{i}, \\
|T_{n}|_{X^{i}}| &\leq &\rho _{i},
\end{eqnarray*}%
where the integers $k_{i}$ satisfy $0=k_{0}<k_{1}<k_{2}<...<+\infty $ and $%
\rho _{i}$ are nonnegative reals. Here $T_{n}|_{X^{i}}$ denotes the
(set-theoretic) restriction of $T_{n}$ to $X^{i}$. Furthermore, we set%
\begin{equation*}
\rho _{\infty }:=\liminf_{s\rightarrow +\infty }\left[ \dprod%
\limits_{i=1}^{s-1}\rho _{i}^{(k_{i+1}-k_{i})}\right] ^{\frac{1}{k_{s}}},
\end{equation*}%
and assume 
\begin{equation*}
\rho _{\infty }<+\infty .
\end{equation*}
\end{condition}

Recall the notion of compactness (and closely related complete continuity)
of an operator as taking bounded sets to relatively compact sets. Requiring $%
\rho _{\infty }=0^{+}$ in the above condition is equivalent to \textit{%
uniform compactness} for operators in Banach spaces, cf. \cite{kurz1, serre}%
. Note also that, optimally defined, the sequence $\rho _{i}$ should be
decreasing. However, as we deal with the nonautonomous case, we consider
these bounds to be given extrinsincally instead of intrinsically (contrast
Lyapunov exponents in the autonomous case).

\subsection{Lemmata and proof of main Theorem}

\begin{definition}
\label{Definition7.19}Assume $V$ is a $\mathbb{K-}$linear normed space. If $%
\mathbb{K=}%
\mathbb{R}
$ or $%
\mathbb{C}
$, we write $S_{V}:=\{v\in V:|v|=1\}$. For general $\mathbb{K}$, there
exists a positive constant $\Theta $ (which may depend on $\mathbb{K}$),
such that 
\begin{equation}
\forall v\in V\backslash \{0\},|v|<\Theta ,\exists k\in \mathbb{K}^{\times
},|k|\geq \frac{1}{\Theta }\text{ s.t. }1\geq |kv|\geq \Theta  \label{12.25}
\end{equation}%
The value group $|\mathbb{K}^{\times }|$ is either dense in $(0,+\infty )$
(so we may take $\Theta $ to be any constant smaller than $1$, and
arbitrarily close to $1)$, or discrete and there exists a unique constant $%
\Theta \in (0,1)$ s.t. $|\mathbb{K}^{\times }|=\Theta ^{%
\mathbb{Z}
}$. We set%
\begin{equation*}
S_{V}:=\{v\in V:1\geq |v|\geq \Theta \}.
\end{equation*}
\end{definition}

\begin{definition}
Let $X,Y$ be finite dimensional normed $\mathbb{K-}$vector spaces, of the
same dimension $n$, and $\Lambda ^{n}X,\Lambda ^{n}Y$ denote the one
dimensional space of $n-$linear alternating ($\mathbb{K-}$valued) forms on $%
X,Y$ respectively, with the following norms:%
\begin{eqnarray*}
|\alpha | &=&\sup_{y_{i}\in Y,|y_{i}|\leq 1}|\alpha
(y_{1},y_{2},...,y_{n)}|,\forall \alpha \in \Lambda ^{n}Y, \\
|\alpha | &=&\sup_{x_{i}\in X,|x_{i}|\leq 1}|\alpha
(x_{1},x_{2},...,x_{n)}|,\forall \alpha \in \Lambda ^{n}X.
\end{eqnarray*}%
For any continuous linear map $T:X\rightarrow Y$, consider the induced
pullback map on alternating forms 
\begin{eqnarray*}
T^{\ast } &:&\Lambda ^{n}Y\rightarrow \Lambda ^{n}X \\
T^{\ast }(\alpha ) &=&\alpha (Tx_{1},Tx_{2},...,Tx_{n)},\forall \alpha \in
\Lambda ^{n}Y.
\end{eqnarray*}%
We define as \textquotedblleft determinant\textquotedblright\ $\det T$ the
norm of this transformation from $\Lambda ^{n}Y$ to $\Lambda ^{n}X$: 
\begin{equation*}
\det T=\frac{|T^{\ast }(\alpha )|}{|\alpha |},\alpha \neq 0.
\end{equation*}
\end{definition}

\begin{lemma}
\label{Lemma3.48}The determinant defined in the above definition enjoys the
following properties:

\begin{itemize}
\item if $|Tx|=|x|,\forall x\in X$, then $\det T=1.$ In particular, when $%
T=id_{X}$ we have $\det id_{X}=1$.

\item multiplicativity: for $T:X\rightarrow Y$ and $S:Y\rightarrow Z$ we
have 
\begin{equation*}
\det S\circ T=\det S\det T.
\end{equation*}
\end{itemize}
\end{lemma}

\begin{proof}
If $|Tx|=|x|,\forall x\in X$ is satisfied, then the mapping is obviously
injective, hence bijective. In particular, $T\left( \{x\in X:|x|\leq
1\}\right) =\{y\in Y:|y|\leq 1\}$, which by definition gives $|T^{\prime
}(\alpha )|=|\alpha |,\forall \alpha $. It suffices to prove
multiplicativity in the injective case (otherwise, the fact is trivial). We
have by definition 
\begin{eqnarray*}
&&\det S\det T \\
&=&\frac{\sup_{y_{i}\in Y,|y_{i}|\leq 1}|\alpha (S\left( y_{1}\right)
,S\left( y_{2}\right) ,...,S(y_{n}))|}{\sup_{z_{i}\in Z,|z_{i}|\leq
1}|\alpha (z_{1},z_{2},...,z_{n})|}\frac{\sup_{x_{i}\in X,|x_{i}|\leq
1}|\beta (T\left( x_{1}\right) ,T\left( x_{2}\right) ,...,T\left(
x_{n}\right) )|}{\sup_{y_{i}\in Y,|y_{i}|\leq 1}|\beta
(y_{1},y_{2},...,y_{n})|},
\end{eqnarray*}%
for arbitrary nontrivial $\alpha ,\beta $. Now set $\beta \in \Lambda ^{n}Y$
to be $\beta (y_{1},y_{2},...,y_{n}):=\alpha (S\left( y_{1}\right) ,S\left(
y_{2}\right) ,...,S(y_{n}))$. Then substituting we get 
\begin{eqnarray*}
&&\det S\det T \\
&=&\frac{\sup_{x_{i}\in X,|x_{i}|\leq 1}|\alpha (S(T\left( x_{1}\right)
)),\alpha (S(T\left( x_{2}\right) )),...,\alpha (S(T\left( x_{n}\right) ))|}{%
\sup_{z_{i}\in Z,|z_{i}|\leq 1}|\alpha (z_{1},z_{2},...,z_{n})|} \\
&=&\det S\circ T.
\end{eqnarray*}
\end{proof}

\begin{lemma}[\protect\cite{kurz1, kurz73}]
\label{Lemma3.43} Let $(a_{ij})$ stand for the $n\times n$ matrix with
entries $a_{ij}\in \mathbb{K}$, and $\det (a_{ij})$ its standard
determinant.\ We define%
\begin{equation*}
g(\theta ,n):=\sup_{|a_{ij}|\leq \theta }|\det (a_{ij})|.
\end{equation*}%
Obviously 
\begin{equation}
g(\theta ,n):=\sup_{|a_{ij}|\leq \theta }|\det (a_{ij})|\leq \theta ^{n}n!.
\label{1.52}
\end{equation}%
In particular, when $\mathbb{K=}%
\mathbb{R}
$, this function satisfies the Hadamard bound 
\begin{equation}
g(1,n)\leq n^{\frac{n}{2}},  \label{1.53}
\end{equation}%
and we further have 
\begin{equation*}
\lim_{n\rightarrow \infty }\frac{g(1,n)}{n^{\frac{n}{2}}}=1.
\end{equation*}
\end{lemma}

While the classical Hahn-Banach extension property is well-known, it is
worthwhile to recall, that it holds under no additional assumptions for
linear spaces over arbitrary fields, if one weakens the statement by an
\textquotedblleft epsilon\textquotedblright . In fact, it holds even for the
off-the-beaten-track case of a normed vector space whose norm satisfies
sometimes the standard triangle (archimedean) inequality, and other times
the ultrametric (nonarchimedean) inequality, over a nonarchimedean field.
This fact was remarked by Monna \cite{monnaa,monna}. The following proof
follows \cite{monnaa,monna, schn}.

\begin{theorem}
\label{Theorem10.48}Let $V$ be a normed vector space over the field $\mathbb{%
K}$, and $W\subset V$ a vector subspace. Assume that $V$ has a finite or at
most countable dense subspace. Let $f:W\rightarrow \mathbb{K}$ be a bounded
linear form, of norm $|f|=\sup_{x\in W\backslash \{0\}}\frac{|f(x)|}{|x|}$.
Then, for any $\varepsilon >0$, there exists a linear extension $f^{\prime
}:V\rightarrow \mathbb{K}$, of norm $|f^{\prime }|=\sup_{x\in V\backslash
\{0\}}\frac{|f^{\prime }(x)|}{|x|}\leq (1+\varepsilon )|f|$, with $f^{\prime
}|_{W}=f$.
\end{theorem}

\begin{proof}
As $V$ possesses an at most countable dense subspace, it suffices to prove
the fact for the extension by a single dimension. Consider a vector $v\in
V\backslash W$. We wish to linearly extend $f$ to $W+v$. The infimum%
\begin{equation*}
\inf \{|v+w|,w\in W\}
\end{equation*}%
is either zero or positive. If it is zero, then we may simply extend $f$
continuously. If it is positive, it may be attained or not, but in any case,
there exists $w_{v}\in W$ such that 
\begin{equation}
|v+w_{v}|\leq (1+\varepsilon )|v+w|,\forall w\in W\text{.}  \label{9.44}
\end{equation}%
We set 
\begin{equation*}
f(v)=f(w_{v})\text{. }
\end{equation*}%
We obtain by linearity%
\begin{equation*}
|f(kv+w)|=|k||f(k^{-1}w)+f(w_{v})|=|k||f(k^{-1}w+w_{v})|.
\end{equation*}%
Now by definition of the norm of $f$ we get 
\begin{equation*}
|f(kv+w)|\leq |k||k^{-1}w+w_{v}||f|.
\end{equation*}%
Finally, applying \eqref {9.44} we conclude that 
\begin{equation*}
|f(kv+w)|\leq |k|(1+\varepsilon )|k^{-1}w+v||f|=|kv+w||f|(1+\varepsilon ).
\end{equation*}
\end{proof}

\begin{lemma}[{\protect\cite{kurz72, kurz1}, cf. \protect\cite[Proposition
9.1.5]{pietsch}}]
\label{Lemma3.44} Let $\varkappa \in (0,1]$ be fixed. Let $V$ be a finite
dimensional normed $\mathbb{K-}$vector space, of dimension $m\geq 1$, and
assume the elements $v_{i}\in V$ satisfy. 
\begin{eqnarray}
1 &\geq &|v_{i}|, \\
|v_{i}+\sum_{j>i}\lambda _{j}v_{j}| &\geq &\varkappa ,  \label{10.48}
\end{eqnarray}%
for all $i=1,...,m$ and any $\lambda _{j}\in \mathbb{K}$. Then, for any
fixed $\varepsilon >0$, there exist $\eta _{ij}\in \mathbb{K}$ s.t. setting 
\begin{equation}
u_{i}=v_{i}+\sum_{j>i}\eta _{ij}v_{j}  \label{9.40}
\end{equation}%
we have 
\begin{equation}
\{v\in V:|v|\leq 1\}\subset \left\{ \sum_{k=1}^{m}\mu _{k}u_{k}:\mu _{k}\in 
\mathbb{K}\text{ and }|\mu _{k}|\leq \frac{1}{\varkappa }(1+\varepsilon
)\right\} .  \label{11.00}
\end{equation}
\end{lemma}

\begin{proof}
In virtue of \eqref {10.48} and Theorem {\ref{Theorem10.48} there exist
linear functionals }$f_{k}:V\rightarrow \mathbb{K}$ such that%
\begin{eqnarray}
|f_{k}| &\leq &\frac{1}{\varkappa }(1+\varepsilon )  \label{10.59} \\
f_{k}(v_{k}) &=&\frac{1}{\varkappa },  \notag \\
f_{k}(v_{j}) &=&0,1\leq k<j\leq m.  \label{11.35}
\end{eqnarray}%
One may find $\eta _{ij}\in \mathbb{K}$ so that 
\begin{equation}
f_{k}(v_{i}+\sum_{j\geq i+1}\eta _{ij}v_{j})=0,i<k\leq m.  \label{11.34}
\end{equation}%
This is easy to see. Fix $i$. We set $\eta _{i(i+1)}$ so that $%
f_{i+1}(v_{i}+\eta _{i(i+1)}v_{i+1})=0$. \ Next we set $\eta _{i(i+2)}$ so
that $f_{i+2}(v_{i}+\eta _{i(i+1)}v_{i+1}+\eta _{i(i+2)}v_{i+2})=0$, and
continue inductively. Now \eqref {11.34} holds as a consequence of \eqref
{11.35}.

Then with $u_{i}$ as in \eqref {9.40} we have 
\begin{equation*}
f_{k}(u_{i})=\left\{ 
\begin{array}{cc}
1, & k=i \\ 
0, & k\neq i%
\end{array}%
\right. .
\end{equation*}%
In other words, 
\begin{equation*}
y=\sum_{k=1}^{m}f_{k}(y)u_{k},\forall y\in V.
\end{equation*}%
Now assume that $y\notin \left\{ \sum_{k=1}^{m}\mu _{k}u_{k}:\mu _{k}\in 
\mathbb{K}\text{ and }|\mu _{k}|\leq \frac{1}{\varkappa }(1+\varepsilon
)\right\} $. Then for at least one index $k$, we have $f_{k}(y)>\frac{1}{%
\varkappa }(1+\varepsilon )$, which by \eqref {10.59} gives $|y|>1$. This
proves \eqref {11.00}.
\end{proof}

\begin{lemma}[{\protect\cite{kurz72, kurz1}, cf. \protect\cite[Proposition
9.1.5]{pietsch}}]
\label{Lemma3.45} The determinant $\det T$ for $T:X\rightarrow Y$ satisfies
the folllowing property. If $V_{1}\supset V_{2}\supset ...\supset V_{n}$ are
subspaces of $X$ with $\dim V_{j}=n-j+1$ and 
\begin{equation}
|Tx|\leq \kappa _{j}|x|,x\in V_{j},  \label{3.26}
\end{equation}%
then 
\begin{equation*}
\det T\leq \left[ \liminf_{\varepsilon \rightarrow 0^{+}}g(\frac{1}{\Theta }%
(1+\varepsilon ),n)\right] \dprod\limits_{j=1}^{n}\kappa _{j}.
\end{equation*}
\end{lemma}

\begin{proof}
Set 
\begin{equation*}
\chi (\alpha )=\sup_{v_{i}\in V_{i},|v_{i}|\leq 1}|\alpha
(v_{1},v_{2},...,v_{n)}|,\forall \alpha \in \Lambda ^{n}X.
\end{equation*}%
Notice that $\chi (\cdot )$ is a norm on $\Lambda ^{n}X$ and by definition $%
\chi (\alpha )\leq |\alpha |=\sup_{x_{i}\in X,|x_{i}|\leq 1}|\alpha
(x_{1},x_{2},...,x_{n)}|,\forall \alpha \in \Lambda ^{n}X$. Fix some
nontrivial $\alpha \in \Lambda ^{n}X$. Also consider a $\theta \in (0,1)$ so
that \eqref {12.25} holds. Now consider an arbitrary $\epsilon >0$,
sufficiently small so that $\frac{(1-\epsilon )}{\theta }>1$, and a family
of vectors $y_{i}\in V_{i},|y_{i}|\leq 1,i=1,...,n$ so that 
\begin{equation}
|\alpha (y_{1},y_{2},...,y_{n})|\geq (1-\epsilon )\sup_{x_{i}\in
V_{i},|x_{i}|\leq 1}|\alpha (x_{1},x_{2},...,x_{n})|.  \label{12.00}
\end{equation}%
Then by definition of an alternating form, \eqref {12.25}, \eqref {12.00},
and $\frac{(1-\epsilon )}{\theta }>1,$ we get 
\begin{equation*}
|y_{i}+\sum_{j>i}\lambda _{j}y_{j}|\geq \theta ,
\end{equation*}%
for all $i=1,...,n$ and any $\lambda _{j}\in \mathbb{K}$. We now apply Lemma 
{\ref{Lemma3.44}, and for a fixed }$\varepsilon >0,$ we set 
\begin{equation}
u_{i}=y_{i}+\sum_{j>i}\eta _{ij}y_{j}  \label{1.57}
\end{equation}%
so that \eqref {11.00} holds. We have defining $\eta _{ij}=0,j\leq i$ $\ $%
that $\det (I+(\eta _{ij}))=1$ so $|\alpha (y_{1},y_{2},...,y_{n})|=|\alpha
(u_{1},u_{2},...,u_{n})|$. Notice that $u_{i}\in V_{i}$, thus \eqref{11.00}
implies $(z_{1},z_{2},...,z_{n})=(\omega _{ij})(u_{1},u_{2},...,u_{n})$ with 
$|\omega _{ij}|\leq \frac{1}{\theta }(1+\varepsilon ),$ for any collection $%
z_{i}\in X,|z_{i}|\leq 1$. Using \eqref {1.57} together with the definition
of $\chi (\alpha )$ and the family $y_{i},i=1,...,n$, we obtain the
following inequality 
\begin{eqnarray*}
&&|\alpha (z_{1},z_{2},...,z_{n})| \\
&=&|\det (\omega _{ij})||\alpha (u_{1},u_{2},...,u_{n})| \\
&=&|\det (\omega _{ij})||\alpha (y_{1},y_{2},...,y_{n})| \\
&\leq &g(\frac{1}{\theta }(1+\varepsilon ),n)\chi (\alpha ).
\end{eqnarray*}%
Taking the supremum we get 
\begin{equation*}
|\alpha |\leq g(\frac{1}{\theta }(1+\varepsilon ),n)\chi (\alpha ).
\end{equation*}%
Now for the case a valuation group dense in $(0,+\infty )$ we may take $%
\frac{1}{\theta }(1+\varepsilon )$ arbitrarily close to $1$, and hence
conclude $|\alpha |\leq g(1,n)\chi (\alpha )$. As for the case of a discrete
valuation group, we take $\theta =\Theta $ where $\Theta ^{%
\mathbb{Z}
}$ is the valuation group, concluding 
\begin{equation}
|\alpha |\leq \left[ \liminf_{\varepsilon \rightarrow 0^{+}}g(\frac{1}{%
\Theta }(1+\varepsilon ),n)\right] \chi (\alpha ).  \label{3.29}
\end{equation}

It is easy to bound $\det T$ as follows, using \eqref {3.26} and \eqref{3.29}%
:%
\begin{eqnarray*}
&&\det T=\frac{|T^{\ast }(\alpha )|}{|\alpha |}\leq \frac{\left[
\liminf_{\varepsilon \rightarrow 0^{+}}g(\frac{1}{\Theta }(1+\varepsilon ),n)%
\right] \chi (T^{\ast }(\alpha ))}{\sup_{y_{i}\in Y,|y_{i}|\leq 1}|\alpha
(y_{1},y_{2},...,y_{n)}|} \\
&=&\frac{\left[ \liminf_{\varepsilon \rightarrow 0^{+}}g(\frac{1}{\Theta }%
(1+\varepsilon ),n)\right] \sup_{x_{i}\in V_{i},|x_{i}|\leq
1}|a(Tx_{1},...,Tx_{n})|}{\sup_{y_{i}\in Y,|y_{i}|\leq 1}|\alpha
(y_{1},y_{2},...,y_{n)}|} \\
&=&\frac{\left[ \liminf_{\varepsilon \rightarrow 0^{+}}g(\frac{1}{\Theta }%
(1+\varepsilon ),n)\right] \left( \dprod\limits_{j=1}^{n}\kappa _{j}\right)
\sup_{x_{i}\in V_{i},|x_{i}|\leq 1}|a(\frac{Tx_{1}}{\kappa _{1}},...,\frac{%
Tx_{n}}{\kappa _{n}})|}{\sup_{y_{i}\in Y,|y_{i}|\leq 1}|\alpha
(y_{1},y_{2},...,y_{n)}|} \\
&\leq &\left[ \liminf_{\varepsilon \rightarrow 0^{+}}g(\frac{1}{\Theta }%
(1+\varepsilon ),n)\right] \left( \dprod\limits_{j=1}^{n}\kappa _{j}\right) .
\end{eqnarray*}
\end{proof}

\begin{remark}
The quantity $\left[ \liminf_{\varepsilon \rightarrow 0^{+}}g(\frac{1}{%
\Theta }(1+\varepsilon ),n)\right] $ in the above Theorem may be replaced by 
$g(1,n)$, for fields with a value group dense in $(0,+\infty )$, and by $g(%
\frac{1}{\Theta },n)$ for spherically complete fields \cite{dr,monna,schn,vr}%
.
\end{remark}

\begin{definition}
\bigskip Let $m,p$ be positive integers, and $s$ be a nonnegative integer
such that 
\begin{equation*}
m=pk_{s}+r,
\end{equation*}%
where 
\begin{equation*}
0\leq r\leq p(k_{s+1}-k_{s}).
\end{equation*}%
Set 
\begin{equation*}
\Xi (m,p)=\left[ \liminf_{\varepsilon \rightarrow 0^{+}}g(\frac{1}{\Theta }%
(1+\varepsilon ),m)\right] \rho _{s}^{pr}\dprod\limits_{i=1}^{s-1}\rho
_{i}^{p^{2}(k_{i+1}-k_{i})}.
\end{equation*}
\end{definition}

\begin{lemma}[\protect\cite{kurz1}]
\label{Lemma3.46}\bigskip Let $m,p$ be positive integers, and Condition {\ref%
{Condition4.55} hold}. Let $X$ be a $m$ dimensional subspace of $Y_{0}$.
There exists a constant $\Upsilon $ which only depends on $\mathbb{K}$ and
the sequences of $k,\rho $ of Condition {\ref{Condition4.55}, such that we
have }%
\begin{equation*}
\det T^{n}|_{X}\leq \Upsilon \left[ \Xi (m,p)\right] ^{\frac{n}{p}}.
\end{equation*}
\end{lemma}

\begin{proof}
Suffice to consider the injective case. Write $n+1=p\ell +r$, for
nonnegative integers $\ell ,r$, where $0\leq r\leq p$. We set 
\begin{eqnarray*}
S_{k} &=&T_{kp-1}\circ T_{kp-1}\circ ...\circ T_{k\left( p-1\right) } \\
X_{k} &=&S_{k}\circ ...S_{1}(X)
\end{eqnarray*}%
From Lemma {\ref{Lemma3.45} we get }%
\begin{eqnarray}
\det S_{k}|_{X_{K-1}} &\leq &\Xi (m,p),k=1,...,\ell  \label{1.42} \\
\det T_{n}\circ T_{n-1}\circ T_{p\ell }|_{X_{\ell }} &\leq &\Xi (m,r).
\label{1.43}
\end{eqnarray}%
Applying Lemma {\ref{Lemma3.48} to }\eqref {1.42}, \eqref {1.43} we get 
\begin{equation*}
\det T^{n}|_{X}\leq \Xi (m,r)\left( \Xi (m,p)\right) ^{\ell }.
\end{equation*}%
Now suffice to set 
\begin{equation*}
\Upsilon =\sup_{0\leq r\leq p}\Xi (m,r)\left( \Xi (m,p)\right) ^{-\frac{r-1}{%
p}}.
\end{equation*}
\end{proof}

The following result is known when $\mathbb{K=}%
\mathbb{R}
$, see \cite{kurz1}. We verify that the difference between \eqref {1.52} and %
\eqref {1.53} is inessential.

\begin{lemma}
\label{Lemma3.47}\bigskip Fix $\mathbb{K}$. Let Condition {\ref%
{Condition4.55} hold}, and $\varpi \in (\rho _{\infty },+\infty )$. Then
there exist integers $m,p\geq 1$ such that 
\begin{equation}
\left[ \Xi (m,p)\right] ^{\frac{1}{mp}}<\varpi .  \label{2.20}
\end{equation}%
In particular, for large enough $p$, and $m=pk_{s}$, the ratio%
\begin{equation*}
\frac{\left[ \Xi (m,p)\right] ^{\frac{1}{mp}}}{\left[ \dprod%
\limits_{i=1}^{s-1}\rho _{i}^{(k_{i+1}-k_{i})}\right] ^{\frac{1}{k_{s}}}}
\end{equation*}%
is arbitrarily close to $1$.
\end{lemma}

\begin{proof}
Set $m=pk_{s}$. Then%
\begin{equation*}
\left[ \Xi (m,p)\right] ^{\frac{1}{mp}}=\left[ \liminf_{\varepsilon
\rightarrow 0^{+}}g(\frac{1}{\Theta }(1+\varepsilon ),pk_{s})\right] ^{\frac{%
1}{p^{2}k_{s}}}\left[ \dprod\limits_{i=1}^{s-1}\rho _{i}^{(k_{i+1}-k_{i})}%
\right] ^{\frac{1}{k_{s}}}.
\end{equation*}%
By the Stirling approximation, we have for any fixed $\varepsilon >0$ 
\begin{equation*}
\left[ \dprod\limits_{i=1}^{s-1}\rho _{i}^{(k_{i+1}-k_{i})}\right] ^{\frac{1%
}{k_{s}}}\leq \left[ \Xi (m,p)\right] ^{\frac{1}{mp}}\leq (1+o(1))\left( 
\frac{pk_{s}}{\Theta e}(1+\varepsilon )\right) ^{\frac{1}{p}}\left[
\dprod\limits_{i=1}^{s-1}\rho _{i}^{(k_{i+1}-k_{i})}\right] ^{\frac{1}{k_{s}}%
},
\end{equation*}%
where $o(1)\rightarrow 0$ as $pk_{s}\rightarrow \infty $. Now for fixed $s$,
we may take arbitrarily large $p$ so that the desired result holds.
\end{proof}

\begin{theorem}
\bigskip \label{Theorem11.29}

Let Condition {\ref{Condition4.55} hold}, $\varpi \in (\rho _{\infty
},+\infty )$, and $c>0$ be a constant. Fix integers $m,p\geq 1$ such that %
\eqref {2.20} holds, as in Lemma {\ref{Lemma3.47}.} Assume that for every
fixed $N>0$, there exist points $y^{N}\in $ $Y_{0}$ such that%
\begin{equation}
|T^{N}(y^{N})|\geq c\varpi ^{N}|y^{N}|.  \label{12.52}
\end{equation}%
Then for some integer $k\leq m-1$ and constant $\omega \in ((\varpi
)^{-1}\Xi (m,p)^{\frac{1}{mp}},+\infty )$, for any $\chi \in ((\varpi
)^{-1}\Xi (m,p)^{\frac{1}{mp}},\omega )$, we have for infinitely many $N,$
that any vector parallel to $T^{N}(y^{N})$ in $S_{Y_{N+1}}$ is$\ \Theta 
\frac{\chi ^{N}}{\omega ^{N}-\chi ^{N}}-$close to a linear span of $k$
vectors $x_{i}^{N}\in S_{Y_{N+1}},i=1,...,k$, with coefficients of norm
uniformly bounded by $\Theta \frac{1}{\omega ^{N}-\chi ^{N}}$. If it is
known there exist $m-1$ image (parallel to $T^{N}(y^{N})$) vectors $%
x_{i}^{N}\in S_{Y_{N+1}},i=1,...,m-1$ such that 
\begin{equation}
A_{N}:=\inf \left( \left\vert \sum_{i=1}^{m-1}\eta _{i}x_{i}^{N}\right\vert
:\max |\eta _{i}|=1,\eta _{i}\in \mathbb{K}\right) \geq \varkappa \varrho
^{N}  \label{3.56}
\end{equation}%
where $\varrho \in (\varpi ^{-1}\Xi (m,p)^{\frac{1}{mp}},1]$ is a known
constant, then the same conclusion holds with $\Theta \frac{\chi ^{N}}{%
\varkappa \varrho ^{N}-\chi ^{N}}$ (instead of $\Theta \frac{\chi ^{N}}{%
\omega ^{N}-\chi ^{N}}$) and $\Theta \frac{1}{\varkappa \varrho ^{N}-\chi
^{N}}$ (instead of $\Theta \frac{1}{\omega ^{N}-\chi ^{N}}$) respectively,
for any $\chi \in (\varpi ^{-1}\Xi (m,p)^{\frac{1}{mp}},\varrho )$.
\end{theorem}

\begin{proof}
Assume there exist $m$ vectors $y_{i}^{N}\neq 0,i=1,...,m,$ so that %
\eqref{12.52} holds. Then we have%
\begin{equation}
|T^{N}y_{i}^{N}|\geq c\varpi ^{N}|y_{i}^{N}|.  \label{2.51}
\end{equation}%
We set 
\begin{eqnarray*}
X^{N} &:&=\limfunc{span}\{T^{N}y_{i}^{N}\}, \\
Z^{N} &:&=\limfunc{span}\{y_{i}^{N}\}.
\end{eqnarray*}%
Now fix a collection $x_{i}^{N}\in S_{Y_{N+1}}$ s.t. $x_{i}^{N}$ is parallel
to $T^{N}y_{i}^{N},$ and write $z_{i}^{N}$ for their preimages by $T^{N}$. \
We write%
\begin{equation*}
A_{N}:=\inf \left( \left\vert \sum_{i=1}^{m}\eta _{i}x_{i}^{N}\right\vert
:\max |\eta _{i}|=1,\eta _{i}\in \mathbb{K}\right) .
\end{equation*}%
Then for 
\begin{equation*}
x=\sum_{i=1}^{m}r_{i}x_{i}^{N}
\end{equation*}%
we have 
\begin{equation}
|x|=\left\vert \sum_{i=1}^{m}r_{i}x_{i}^{N}\right\vert
=\max_{i=1,...,m}|r_{i}|\left\vert \sum_{k=1}^{m}\frac{r_{k}}{%
\max_{i=1,...,m}|r_{i}|}x_{k}^{N}\right\vert \geq
\max_{i=1,...,m}|r_{i}|A_{N}.  \label{2.52}
\end{equation}%
Assume $A_{N}>0$. We immediately deduce from \eqref {2.51}, \eqref {2.52}, 
\begin{eqnarray}
m|x| &\geq &A_{N}\sum_{i=1}^{m}|r_{i}x_{i}^{N}|\geq A_{N}c\varpi
^{N}|z|,\forall x\in X^{N}  \notag \\
\frac{m}{A_{N}c(\varpi )^{N}}|x| &\geq &|T^{-N}|_{X^{N}}(x)|,\forall x\in
X^{N}\text{.}  \label{2.49}
\end{eqnarray}%
where 
\begin{equation*}
z=\sum_{i=1}^{m}r_{i}z_{i}^{N}.
\end{equation*}%
Then from inequality \eqref {2.49}, Lemmata {\ref{Lemma3.45}, \ref{Lemma3.46}%
, we obtain successively} 
\begin{eqnarray*}
1 &=&\det \left( T^{-N}|_{X^{N}}\circ T^{N}|_{Z^{N}}\right) \\
&=&\det \left( T^{-N}|_{X^{N}})\det (T^{N}|_{Z^{N}}\right) \\
&\leq &\left[ \liminf_{\varepsilon \rightarrow 0^{+}}g(\frac{1}{\Theta }%
(1+\varepsilon ),m)\right] \left[ \frac{m}{A_{N}c\varpi ^{N}}\right]
^{m}\Upsilon \left[ \Xi (m,p)\right] ^{\frac{N}{p}} \\
&=&A_{N}^{-m}\Upsilon m^{m}c^{-m}\left[ \liminf_{\varepsilon \rightarrow
0^{+}}g(\frac{1}{\Theta }(1+\varepsilon ),m)\right] \left\{ (\varpi
)^{-1}\Xi (m,p)^{\frac{1}{mp}}\right\} ^{mN}\text{.}
\end{eqnarray*}%
So $A_{N}$ decays exponentially in $N$, as $\left\{ (\varpi )^{-1}\Xi (m,p)^{%
\frac{1}{mp}}\right\} ^{N}$.

Let $D$ denote the highest dimension for which the exists some constant $%
\omega \in (\varpi ^{-1}\Xi (m,p)^{\frac{1}{mp}},+\infty )$ s.t. for
infinitely many $N>0$, we have for some such $A_{N}$, constructed as above
for each $N$, that $A_{N}\geq \omega ^{N}$. Then by the above calculations
we must have $1\leq D\leq m-1$ (the case $D=0$, that is when there are no
vectors st \eqref {12.52} holds, is trivial). Fix such a sequence of $N,$
and collections $x_{i}^{N},i=1,...,d$, so that%
\begin{equation*}
\inf \left( \left\vert \sum_{i=1}^{D}\eta _{i}x_{i}^{N}\right\vert :\max
|\eta _{i}|=1,\eta _{i}\in \mathbb{K}\right) \geq \omega ^{N}.
\end{equation*}%
Assume consider another vector $\psi ^{N}$, so that together with $%
x_{i}^{N},i=1,...,D$, we again have $\psi ^{N}\in S_{Y_{N+1}}$ and parallel
to an image $T^{N}y^{N}$, so that \eqref {12.52} holds. Now for appropriate $%
\eta _{i}(N),\mu (N)\in \mathbb{K}$ with $\max_{i=1,...,D}|\eta
_{i}(N)|,|\mu (N)|=1,$ by the definition of $D,$ the quantity $\left\vert
\left( \sum_{i=1}^{D}\eta _{i}(N)x_{i}^{N}\right) +\mu (N)\psi
^{N}\right\vert $ can be made $\leq \chi ^{N}$, \ for any $\chi \in (\theta
^{-1}\Xi (m,p)^{\frac{1}{mp}},\omega )$. Hence we obtain the following
inequality%
\begin{eqnarray*}
\left\vert \mu (N)\psi ^{N}\right\vert &\geq &\left\vert \sum_{i=1}^{D}\eta
_{i}(N)x_{i}^{N}\right\vert -\left\vert \left( \sum_{i=1}^{D}\eta
_{i}(N)x_{i}^{N}\right) +\mu (N)\psi ^{N}\right\vert \\
&\geq &\omega ^{N}-\chi ^{N}.
\end{eqnarray*}%
We deduce 
\begin{equation*}
\left\vert \mu (N)\right\vert \geq \frac{1}{\Theta }\left( \omega ^{N}-\chi
^{N}\right) ,
\end{equation*}%
concluding%
\begin{equation*}
\left\vert \left( \sum_{i=1}^{D}\frac{\eta _{i}(N)}{\mu (N)}x_{i}^{N}\right)
+\psi ^{N}\right\vert \leq \Theta \frac{\chi ^{N}}{\omega ^{N}-\chi ^{N}}.
\end{equation*}
\end{proof}

\subsection{\protect\bigskip Remarks on nonlinear systems}

The calculations of the above Theorem can easily be adapted to either
forward or backwards evolution, eventually regarding inductive or projective
limits (cf. \cite{dkps}). For the nonlinear case, that is, when $T$
represents the derivative of a nonlinearity, it is plausible we have that
for every fixed $N>0$, if $|x-y|\leq \delta $,$x,y\in K$ for $\delta (N)$
small enough, then 
\begin{equation*}
|T^{N}(x-y)|\geq c\varpi ^{N}|x-y|.
\end{equation*}%
holds, where $K$ $\subset Y_{0}$. If one can relate the $\delta $ and $N,$
which depend on the higher derivatives of the nonlinearity, then it is
possible to calculate dimension estimates for $K$, or related inductive
limits.

\begin{lemma}
\label{Lemma4.45}Assume $f_{N}:Y_{N}\rightarrow Y_{N+1}$ is a nonlinearity
such that 
\begin{eqnarray}
|f_{N}(x)-f_{N}(y)-T_{N}(x-y)| &\leq &M|x-y|^{\Lambda },  \label{1.16} \\
|f_{N}(x)-f_{N}(y)-T_{N}(x-y)| &\leq &L|x-y|.  \label{1.17}
\end{eqnarray}%
for some positive constants $L,M>0$ and $\Lambda >1$. Let us write $%
f^{N}:=f_{N}\circ ...\circ f_{0}$. \ Assume $|T_{N}|=\rho _{0}\leq C$ holds
uniformly. Then for fixed $\delta >0,$ there exists a constant $\gamma >0$
st if for arbitrarily large $N$ we have 
\begin{equation}
|x-y|\leq \gamma \frac{A^{\frac{N}{\Lambda -1}}}{\left( L+C\right) ^{\frac{%
\Lambda }{\Lambda -1}N}}\text{ and }|f^{N}(x)-f^{N}(y)|\geq \delta A^{N}|x-y|
\label{4.25}
\end{equation}%
then we may conclude (for any fixed $\delta ^{\prime }\in (0,\delta )$) 
\begin{equation*}
|T^{N}(x-y)|\geq \delta ^{\prime }A^{N}|x-y|.
\end{equation*}%
Furthermore, $|f^{N}(x)-f^{N}(y)-T^{N}\left( x-y\right) |$ is bounded by $%
\eta \left( L+C\right) ^{N\Lambda }|x-y|^{\Lambda }$ for $N\geq N_{0}$ where 
$\eta $ is the constant%
\begin{equation*}
\sup_{N\geq N_{0}}\frac{MC^{N}\left( 1+\frac{\left( L+C\right) ^{\Lambda }}{C%
}\frac{\left[ \frac{\left( L+C\right) ^{\Lambda }}{C}\right] ^{N}-1}{\frac{%
\left( L+C\right) ^{\Lambda }}{C}-1}\right) }{\left( L+C\right) ^{N\Lambda }}%
.
\end{equation*}%
.
\end{lemma}

\begin{proof}
By assumption we have 
\begin{equation}
|f^{N}(x)-f^{N}(y)|\leq (L+C)^{N+1}|x-y|.  \label{1.21}
\end{equation}%
Set 
\begin{equation*}
e_{N}=|f^{N}(x)-f^{N}(y)-T^{N}\left( x-y\right) |.
\end{equation*}%
Then \eqref {1.16}, \eqref {1.17}, imply that%
\begin{equation}
e_{N+1}=|f^{N+1}(x)-f^{N+1}(y)-T^{N+1}\left( x-y\right) |\leq
M|f^{N}(x)-f^{N}(y)|^{\Lambda }+Ce_{N}.  \label{1.22}
\end{equation}%
We now apply wellknown techniques to solve the first order difference
inequality \eqref {1.22}, obtaining%
\begin{eqnarray*}
\frac{e_{N+1}}{C^{N+1}} &\leq &\frac{1}{C^{N+1}}M|f^{N}(x)-f^{N}(y)|^{%
\Lambda }+\frac{e_{N}}{C^{N}}, \\
\frac{e_{N}}{C^{N}} &\leq &e_{0}+M\sum_{k=0}^{N-1}\frac{1}{C^{k+1}}%
|f^{k}(x)-f^{k}(y)|^{\Lambda }.
\end{eqnarray*}%
Using \eqref {1.16}, \eqref {1.21}, and the definition of $e_{N}$, the
following computations are then immediate:%
\begin{eqnarray*}
\frac{e_{N}}{C^{N}} &\leq &e_{0}+M\sum_{k=0}^{N-1}\frac{1}{C^{k+1}}%
|f^{k}(x)-f^{k}(y)|^{\Lambda } \\
&\leq &M|x-y|^{\Lambda }+M\sum_{k=0}^{N-1}\frac{1}{C^{k+1}}%
|f^{k}(x)-f^{k}(y)|^{\Lambda } \\
&\leq &M|x-y|^{\Lambda }+M\sum_{k=0}^{N-1}\frac{1}{C^{k+1}}(L+C)^{\Lambda
k+\Lambda }|x-y|^{\Lambda } \\
&\leq &M|x-y|^{\Lambda }\left( 1+\sum_{k=0}^{N-1}\frac{1}{C^{k+1}}%
(L+C)^{\Lambda k+\Lambda }\right) \\
&\leq &M|x-y|^{\Lambda }\left( 1+\sum_{k=0}^{N-1}\left[ \frac{\left(
L+C\right) ^{\Lambda }}{C}\right] ^{k+1}\right) \\
&\leq &M|x-y|^{\Lambda }\left( 1+\frac{\left( L+C\right) ^{\Lambda }}{C}%
\frac{\left[ \frac{\left( L+C\right) ^{\Lambda }}{C}\right] ^{N}-1}{\frac{%
\left( L+C\right) ^{\Lambda }}{C}-1}\right) .
\end{eqnarray*}%
We may conclude 
\begin{equation*}
e_{N}\leq MC^{N}\left( 1+\frac{\left( L+C\right) ^{\Lambda }}{C}\frac{\left[ 
\frac{\left( L+C\right) ^{\Lambda }}{C}\right] ^{N}-1}{\frac{\left(
L+C\right) ^{\Lambda }}{C}-1}\right) |x-y|^{\Lambda }.
\end{equation*}%
Notice that as $\left( 1+\frac{\left( L+C\right) ^{\Lambda }}{C}\frac{\left[ 
\frac{\left( L+C\right) ^{\Lambda }}{C}\right] ^{N}-1}{\frac{\left(
L+C\right) ^{\Lambda }}{C}-1}\right) $ is asymptotic to $\left[ \frac{\left(
L+C\right) ^{\Lambda }}{C}\right] ^{N}$, the right hand side of \eqref {4.26}
is asymptotic to $\left( L+C\right) ^{N\Lambda }|x-y|^{\Lambda }$.
Substituting \eqref {4.25}%
\begin{equation}
e_{N}\leq |x-y|\gamma MC^{N}\left( 1+\frac{\left( L+C\right) ^{\Lambda }}{C}%
\frac{\left[ \frac{\left( L+C\right) ^{\Lambda }}{C}\right] ^{N}-1}{\frac{%
\left( L+C\right) ^{\Lambda }}{C}-1}\right) \frac{A^{N}}{\left( L+C\right)
^{\Lambda N}}.  \label{4.26}
\end{equation}%
Similarly, the right hand side of \eqref {4.26} can be made smaller than $%
\left( \delta -\delta ^{\prime }\right) A^{N}$.
\end{proof}

\begin{remark}
\label{Remark4.28} In the above Lemma, the derivatives of $f_{N}$ at either
the point $x$ or the point $y$ could fullfill the role of $T_{N}$ in \eqref
{1.16}, \eqref{1.17}. In the Corollary below, time may be considered as
running in the opposite direction, though this is no more than a
re-indexing, considering the maps $T_{n},f_{n}:Y_{n+1}\rightarrow Y_{n},$and 
$T^{n}:T_{0}\circ ...\circ T_{n},f^{n}:f_{0}\circ ...\circ f_{n}$. We shall
be assuming the derivatives at any point satisfy the assumptions of Lemma {%
\ref{Lemma4.45} as well as} Condition {\ref{Condition4.55} for the sequence }%
$f_{N}^{\prime }(x),f_{N}^{\prime }(f_{N+1}(x)),f_{N}^{\prime
}(f_{N+1}(f_{N+2}(x))),...,$ though we will be suppressing the space
dependence for simpler notation.
\end{remark}

\begin{definition}[\protect\cite{pertti, ps}]
Let $A$ be a compact (subset of) metric space. For $\varepsilon >0$, set $%
K(\varepsilon ) $ to be the least number of $\varepsilon $ balls needed to
cover $A$. We define the Minkowski dimension of $A$ to be the quantity%
\begin{equation*}
\dim _{M}A:=\limsup_{\epsilon \rightarrow 0^{+}}\frac{\ln K(\epsilon )}{-\ln
(\epsilon )}.
\end{equation*}
\end{definition}

\begin{corollary}
\bigskip \label{Corollary11.35}Let $\mathbb{K=}%
\mathbb{R}
$ and $G$ be a compact subset of $Y_{0}$. Using the same notation as in
Lemma {\ref{Lemma4.45}, } Remark {\ref{Remark4.28}, }we assume the
hypotheses of Lemma {\ref{Lemma4.45},} Condition {\ref{Condition4.55} }%
(uniformly in the space variable $x$). Let $\varpi \in (\rho _{\infty
},+\infty )$, and fix integers $m,p\geq 1$ such that \eqref {2.20} holds, as
in Lemma {\ref{Lemma3.47}, and further} \eqref {3.56} with constant $\varrho 
$ holds as in Theorem {\ref{Theorem11.29}}(again uniformly). Assume that for
a fixed $\iota >0$, for every $x,y\in G$, we have%
\begin{equation}
\iota \varpi ^{N}|f^{-N}(x)-f^{-N}(y)|\leq |x-y|,N=1,2,...  \notag
\end{equation}%
where $f^{-N}(x)$ denote preimages by the maps $f^{N}$. Then,%
\begin{equation*}
\dim _{M}G\leq \left( m-1\right) \frac{\ln \left[ (\varpi )^{-1}\Xi (m,p)^{%
\frac{1}{mp}}\right] }{\ln \left[ (\varpi )^{-1}\Xi (m,p)^{\frac{1}{mp}}%
\right] -\ln \varrho }.
\end{equation*}
\end{corollary}

\begin{proof}
Fix $\gamma \in (0,1),\chi \in ((\varpi )^{-1}\Xi (m,p)^{\frac{1}{mp}%
},\varrho )$, such that 
\begin{equation}
\gamma ^{\Lambda -1}<\frac{\chi }{\varrho }\text{.}  \label{10.39}
\end{equation}
For arbitrarily large integer $N$ and arbitrarily small $\Delta >0$,
consider a $\Delta \gamma ^{N}\frac{\varpi ^{\left( \frac{\Lambda }{\Lambda
-1}\right) N}}{\left( L+C\right) ^{\left( \frac{\Lambda }{\Lambda -1}\right)
N}}$ cover of $G$ with $P$ balls. Then by assumption we may pull it back to
a $\Delta \frac{\gamma ^{N}}{\iota }\frac{\varpi ^{\frac{N}{\Lambda -1}}}{%
\left( L+C\right) ^{\left( \frac{\Lambda }{\Lambda -1}\right) N}}$ cover of $%
f^{-N}(G)$, with centers $p_{\nu },\nu =1,...,P$.

{By Lemma \ref{Lemma4.45}, }we have that for fixed $p_{\nu }\in f^{-N}(G)$,
the vectors $p_{\nu }-y$, where $y\ $in the preimage of the corresponding
ball are such that $|p_{\nu }-y|\leq \Delta \frac{1}{\iota }\gamma ^{N}\frac{%
\varpi ^{\frac{N}{\Lambda -1}}}{\left( L+C\right) ^{\left( \frac{\Lambda }{%
\Lambda -1}\right) N}}$, satisfy the assumptions of Theorem {\ref%
{Theorem11.29}, and more specifically, } 
\begin{eqnarray*}
&&|f^{N}(p_{\nu })-f^{N}(y)-T^{N}\left( p_{\nu }-y\right) |\leq \eta \left(
L+C\right) ^{\Lambda N}|p_{\nu }-y|^{\Lambda } \\
&\leq &\eta \Delta ^{\Lambda }\frac{1}{\iota ^{\Lambda }}\gamma ^{\Lambda N}%
\frac{\varpi ^{\left( \frac{\Lambda }{\Lambda -1}\right) N}}{\left(
L+C\right) ^{\left( \frac{\Lambda }{\Lambda -1}\right) N}}.
\end{eqnarray*}%
From Theorem {\ref{Theorem11.29}} and our hypotheses, we know there exist $%
m-1$ unit vectors $z_{i}^{N},i=1,...,m-1$ st any other unit vector parallel
to $T^{N}\left( p_{\nu }-y\right) $ will be in an $\frac{\chi ^{N}}{%
\varkappa \varrho ^{N}-\chi ^{N}}$ neighborhood of their linear combinations
with coefficients of magnitude no more than $\frac{1}{\varkappa \varrho
^{N}-\chi ^{N}}$ in $Y_{0}$. The vectors $T^{N}\left( p_{\nu }-y\right) $
will be of magnitude at most$\left[ \Delta +\eta \Delta ^{\Lambda }\frac{1}{%
\iota ^{\Lambda }}\gamma ^{(\Lambda -1)N}\right] \gamma ^{N}\frac{\varpi
^{\left( \frac{\Lambda }{\Lambda -1}\right) N}}{\left( L+C\right) ^{\left( 
\frac{\Lambda }{\Lambda -1}\right) N}}$ (so we may scale down by this
factor).

The dimension bounds now follow from wellknown techniques (cf. \cite%
{cft,eden,eft,gt,mane,th,tha} and \cite[p. 367]{t}).

Taking into account \eqref{10.39}, we now can cover each original ball with
new balls of radius $\varsigma \gamma ^{N}\frac{\varpi ^{\left( \frac{%
\Lambda }{\Lambda -1}\right) N}}{\left( L+C\right) ^{\left( \frac{\Lambda }{%
\Lambda -1}\right) N}}\left( \frac{\chi }{\varrho }\right) ^{N}$ (where $%
\varsigma $ is a constant) centered at the points $(a_{1},...,a_{m-1})=\left[
\Delta +\eta \Delta ^{\Lambda }\frac{1}{\iota ^{\Lambda }}\gamma ^{(\Lambda
-1)N}\right] \gamma ^{N}\frac{\varpi ^{\left( \frac{\Lambda }{\Lambda -1}%
\right) N}}{\left( L+C\right) ^{\left( \frac{\Lambda }{\Lambda -1}\right) N}}%
\frac{\chi ^{N}}{\varkappa \varrho ^{N}-\chi ^{N}}%
\sum_{i=1}^{m-1}a_{i}z_{i}^{N}$, where $a_{i}$ take values $%
0,1,...,\left\lceil \frac{1}{\chi ^{N}}\right\rceil $ and with an
appropriate shift so that $f^{N}(p_{\nu })$ is the origin in $Y_{0}$. For
example, 
\begin{equation*}
\varsigma =100\Delta \left[ 1+\eta 100^{\left( \Lambda -1\right) }\frac{1}{%
\iota ^{\Lambda }}\frac{\gamma }{\left( \frac{\chi }{\varrho }\right) ^{100}}%
^{(\Lambda -1)100}\right] (m+100)\left( \frac{1}{\varkappa }+100\right) ,
\end{equation*}%
will certainly suffice.

We have for large enough $N$ constructed a cover of $G$ with only $P\left(
\left\lceil \frac{1}{\chi ^{N}}\right\rceil +1\right) ^{m-1}$ such balls,
thus we get for some constant $\xi :=\varsigma /\Delta >0$ which does not
depend on $N$ (that fact that we can take this constant to be large enough
but independent of $N$, also means we need not be concerned with the
difference between defining Minkowski dimension with balls whose centers are
in $G$, as opposed to the ambient space $Y_{0}$, cf. \cite{mane}) 
\begin{eqnarray*}
K\left( \varsigma \gamma ^{N}\frac{\varpi ^{\left( \frac{\Lambda }{\Lambda -1%
}\right) N}}{\left( L+C\right) ^{\left( \frac{\Lambda }{\Lambda -1}\right) N}%
}\left( \frac{\chi }{\varrho }\right) ^{N}\right) &\leq &K(\Delta \gamma ^{N}%
\frac{\varpi ^{\left( \frac{\Lambda }{\Lambda -1}\right) N}}{\left(
L+C\right) ^{\left( \frac{\Lambda }{\Lambda -1}\right) N}})\left(
\left\lceil \frac{1}{\chi ^{N}}\right\rceil +1\right) ^{m-1}, \\
K\left( \Delta \gamma ^{N}\frac{\varpi ^{\left( \frac{\Lambda }{\Lambda -1}%
\right) N}}{\left( L+C\right) ^{\left( \frac{\Lambda }{\Lambda -1}\right) N}}%
\left\{ \left( \frac{\chi }{\varrho }\right) ^{N}\xi \right\} \right) &\leq
&K(\Delta \gamma ^{N}\frac{\varpi ^{\left( \frac{\Lambda }{\Lambda -1}%
\right) N}}{\left( L+C\right) ^{\left( \frac{\Lambda }{\Lambda -1}\right) N}}%
)\left( \left\lceil \frac{1}{\chi ^{N}}\right\rceil +1\right) ^{m-1}.
\end{eqnarray*}%
Hence 
\begin{equation*}
\limsup_{\epsilon \rightarrow 0^{+}}\frac{\ln K(\epsilon )}{-\ln (\epsilon )}%
\leq \liminf_{N\rightarrow \infty }\frac{\ln \left( \left\lceil \frac{1}{%
\chi ^{N}}\right\rceil +1\right) ^{m-1}}{|\ln (\xi \left( \frac{\chi }{%
\varrho }\right) ^{N})|}=\frac{(m-1)\ln \chi }{\ln \chi -\ln \varrho }.
\end{equation*}%
Letting $\chi $ tend to $\varpi ^{-1}\Xi (m,p)^{\frac{1}{mp}}$ completes the
proof.
\end{proof}

\subsection{\protect\bigskip General comments and possible generalizations}

\bigskip Corollary {\ref{Corollary11.35}} is reminiscent of the bounds on
admissible perturbations of hyperbolic systems so that hyperbolicity is
preserved (see e.g. \cite{cl} and the references therein). Though note the
slight difference between the derivative in a nonlinear system and an
original linear system considered wrt its nonlinear perturbation.

There is a close relation between $A$ in the proof of Theorem {\ref%
{Theorem11.29}} and other definitions of \textquotedblleft
angle\textquotedblright , such as minimum distance on the unit sphere \cite%
{clark, pasq, mil},\cite[Chapter 1, Section 11]{ms}. Unfortunately, in the
nonarchimedean case, the residue class field does not provide an analog of
the unit sphere that would be useful for describing distances between
directions. For this reason, in Definition {\ref{Definition7.19}} we have
opted for a metric analog, $S_{X}$. What the optimal way of reconciling the
algebraic and metric structures would be, is not obvious.

The difference between autonomous and nonautonomous systems, which can
considered to lie in the existence/attainment of a limit of growth estimates
in the autonomous case (as opposed to a limit superior in the nonautonomous
case), merits further investigation. The usage of lifts (cf. \cite{grs, how,
lms}, \cite[Chapter 3]{cl}) to establish or explicitly calculate the growth
estimates of nonautonomous systems is an interesting challenge. It could be
fruitful to combine the above \textit{a priori }estimates with the knowledge
that a limit exists in the sense of a Lyapunov exponent of an autonomous
lift.

All the above computations heavily relied on the existence of a norm, and
the availability of dynamical information encoded in such. This may be a
drawback if a norm does not exist, or if the structure associated to the
norm does not easily reveal any growth estimates. Locally convex spaces may
be dealt with as limits of Banach spaces, and one should be optimistic
concerning estimates of diametral dimension \cite{kolm}. The results and
techniques developed in e.g. \cite{abdj, bd,d17,d, dfw, dg, ter} could be
useful in this regard. On the other hand, in certain algebraic structures,
e.g. when the valuation of $\mathbb{K}$ is trivial, it may be challenging to
attach quantitative estimates of growth/decay. In such cases it may be
desirable to work with the Berkovich spectrum \cite{bak, br, ben, berk}.
Concerning the nonarchimedean case, it should be noted that similar notions
of determinant or volume were recently explored in relation to
Arakelov-Green functions and the Monge-Amp\`{e}re equation (see e.g. \cite%
{dis, be, bgm,loop} and the references therein).

\section{\protect\bigskip Growth/Decay Estimates for Delay Equations}

\subsection{Framework and bounds for linear equations}

\bigskip We now calculate growth estimates for delay equations, directly
relating the delay to Condition {\ref{Condition4.55}. We will be focusing on
solutions taking values in a finite dimensional Banach (Minkowski) space. As
differentiation in complex Banach space is rather special, for the sake of
simplicity, we will restrict our attention to real Banach space, noting that
much the same holds for the complex case. }

Let $E$ denote a real Banach space, $\tau $ a fixed positive constant, and $%
\mathscr{C}\left( [-\tau ,0];E\right) $ the continuous functions from $%
[-\tau ,0]$ to $E$. We assume that $F:%
\mathbb{R}
\times \mathscr{C}\left( [-\tau ,0];E\right) \rightarrow E$ satisfies the
Caratheodory conditions

\begin{itemize}
\item the function $t\mapsto F(t,0)$ is locally Lebesgue (Bochner)
integrable.

\item $F(\cdot ,x)$ is Lebesgue (Bochner) measurable $\forall x\in %
\mathscr{C}\left( [-\tau ,0];E\right) ,$

\item there exists a locally integrable function $n:%
\mathbb{R}
\rightarrow \lbrack 0,+\infty )$ (defined everywhere) such that $\forall
x,y\in \mathscr{C}\left( [-\tau ,0];E\right) $ we have 
\begin{equation}
|F(t,x)-F(t,y)|\leq n(t)|x-y|.  \label{8.08}
\end{equation}
\end{itemize}

If $x\in \mathscr{C}([a-\tau ,b];E)$, where $b\geq a$, we write $%
x_{t}:=x(t+\theta ),\theta \in \lbrack -r,0],\forall t\in \lbrack a,b]$. By
a \textit{solution} of 
\begin{equation}
x^{\prime }(t)=F(t,x_{t}),  \label{7.41}
\end{equation}%
on $[t_{0},+\infty )$ corresponding to the initial function $\phi \in %
\mathscr{C}\left( [-\tau ,0];E\right) $, we understand the unique function $%
x:[t_{0}-\tau ,+\infty )\rightarrow E$, which\ is locally absolutely
continuous on $[t_{0},+\infty )$, such that $x_{t_{0}}(\cdot )=\phi (\cdot
+t_{0})$ and 
\begin{equation*}
x(t)=x(t_{0})+\int_{t_{0}}^{t}F(s,x_{s})ds
\end{equation*}%
holds for $t\geq t_{0}$. The existence and uniqueness of solutions under the
Caratheodory conditions is well-known \cite{bpdm,dm}. Moreover, under such
assumptions, for each $T>0$ there exists a positive constant $c(T)$ such
that given any initial functions $\varphi ,\psi $, and corresponding
solutions $x,y$ respectively, we have 
\begin{equation}
|x-y|_{[0,T]}+|x^{\prime }-y^{\prime }|_{[0,T]}\leq c(T)|\varphi -\psi |%
\text{.}  \label{6.13}
\end{equation}%
We also write $x_{t}(\phi )\in \mathscr{C}\left( [-\tau ,0];E\right) $ and $%
x(t;\phi )\in E$ for the solution corresponding to initial function $\phi $.

By $Q_{n}$, we will denote the solution operator $Q_{n}:\mathscr{C}\left(
[-\tau ,0];E\right) \circlearrowleft $, which associates to each initial
function $\phi \in \mathscr{C}\left( [-\tau ,0];E\right) $, its
corresponding solution on $[n\tau ,(n+1)\tau ],$ for $n\in 
\mathbb{Z}
$. In this section, we will assume $F(t,\cdot )$ is linear, along with
integral boundedness:%
\begin{equation}
\sup_{t\in 
\mathbb{R}
}\int_{t}^{t+\tau }n(s)ds=r<+\infty \text{,}  \label{11.25}
\end{equation}%
and show that (up to a rescaling of time) it implies 
\begin{equation}
|F(s,\cdot )|\leq 1,\forall s\geq 0.  \label{11.29}
\end{equation}

\begin{lemma}
\label{Lemma5.26}Assume $p:%
\mathbb{R}
\rightarrow \lbrack 0,+\infty )$ is locally Lebesgue integrable. Set 
\begin{eqnarray*}
f(t) &:&=\int_{0}^{t}p,t\in 
\mathbb{R}
\\
g(t) &=&\inf \{s\geq 0\colon f(s)=t\},t\in 
\mathbb{R}
.
\end{eqnarray*}%
Then for any locally Bochner integrable $H:%
\mathbb{R}
\rightarrow E,$ and $a,b\in 
\mathbb{R}
$, the function $\frac{1}{p(g(s))}H(g(s)),s\in \lbrack f(a),f(b)]$ is
integrable, and%
\begin{equation*}
\int_{f(a)}^{f(b)}\frac{1}{p(g(s))}H(g(s))ds=\int_{[a,b]\cap
\{x|g(f(x))=x\}}H(t)dt.
\end{equation*}
\end{lemma}

\begin{proof}
From the similar proof in \cite[Appendix B]{brst}, we have that $g$ has the
property that for any measurable $r$, if $r\circ g$ is well-defined, $r\circ
g$ is measurable. We also have $p(g(s))>0$ and $p|_{\{x|g(f(x))\neq x\}}=0$
a.e. Now the desired result follows easily for simple functions, and taking
the limit completes the proof. Alternatively, one could notice that $f,g$
describe an \textit{isomeasure} between $\left( [a,b]\cap
\{x|g(f(x))=x\},p(t)dt\right) $ and $\left( [f(a),f(b)],dt\right) $, in the
sense of \cite[Chapter 12]{pan}. Then the result would follow from \cite[%
Proposition 12.12]{pan}, whose proof is essentially the same, based on first
considering simple functions, then taking limits.
\end{proof}

\begin{theorem}
Assume \eqref {11.25}, and set 
\begin{eqnarray*}
f(t) &:&=%
\begin{array}{cc}
\int_{0}^{t}n(s)ds, & t\geq 0%
\end{array}
\\
g(t) &=&\inf \{s\geq 0\colon f(s)=t\},t\geq 0.
\end{eqnarray*}%
Assume $x$ is a solution of \eqref {7.41} on $[0,+\infty )$ and that $%
f(+\infty )=+\infty $. Consider the function $\tilde{x}$ defined by 
\begin{equation*}
\tilde{x}(s)=x(t)\text{ if and only if }s=f(t).
\end{equation*}%
Set $\tilde{x}_{s}:=\tilde{x}\left( \theta +s\right) ,\theta \in \lbrack
-r,0]$. Then $\tilde{x}$ is a solution of 
\begin{equation*}
\tilde{x}^{\prime }(s)=\widetilde{F}(s,\tilde{x}_{s}),s\geq r.
\end{equation*}%
The functional $\widetilde{F}$ is defined by 
\begin{equation*}
\widetilde{F}(s,\phi )=\frac{1}{n(g(s))}F\left( g(s),\psi \right)
\end{equation*}%
for $\phi \in \mathscr{C}([-r,0];E)$, where $\psi \left( \xi \right) $ is
defined so that $\phi \left( w\right) =\psi \left( \xi \right) ,w\in \lbrack
-r,0],\xi \in \lbrack -\tau ,0]$ iff $w+s=f(g(s)+\xi )$. It is linear,
Caratheodory, and has operator norm $|\widetilde{F}(s,\cdot )|\leq 1,\forall
s\geq 0$.
\end{theorem}

\begin{proof}
It is immediate that the map $\widetilde{F}(s,\cdot )$ is linear and the
above definition takes continuous $\phi $ to continuous $\psi $, of norm not
greater than that of $\phi $. Let us show it is Caratheodory, similarly to
the proof of \cite[p. 242]{bpdm}, \cite[Lemma 4.4]{dm}.

Consider an interval $[a,b]\subset \lbrack r,+\infty )$. Notice that for a
given pointwise assignment $s\mapsto \phi \in \mathscr{C}([-r,0];E)$, the
map $g(s)$ defines $\psi $ pointwise for $s\in \lbrack a,b],$ which we
consider as a function $[a,b]\longrightarrow \mathscr{C}\left( [-\tau
,0];E\right) $. For time-independent $\phi \in \mathscr{C}([-r,0];E)$,
because $g$ continuous except for countably many points and $f$ \ is
continuous, there exists a sequence of step functions $\nu _{n}\in
L^{1}\left( [a,b];\mathscr{C}\left( [-\tau ,0];E\right) \right) $ which
converge to $\psi $ a.e. in $L^{1}\left( [a,b];\mathscr{C}\left( [-\tau
,0];E\right) \right) $. We may set $\nu _{n}$ to be the $\psi \in
L^{1}\left( [a,b];\mathscr{C}\left( [-\tau ,0];E\right) \right) $
corresponding to $g$ replaced by a sequence of step functions $\xi _{n}$
tending (a.e. pointwise in $s$) to $g$ :%
\begin{equation*}
\phi \left( w\right) =\psi \left( \xi \right) ,w\in \lbrack -r,0],\xi \in
\lbrack -\tau ,0]\text{ iff }w+f(\xi _{n}(s))=f(\xi _{n}(s)+\xi ).
\end{equation*}%
Because $F$ is Caratheodory, we have $\lim_{n\rightarrow \infty }F\left(
g(s),\nu _{n}\right) =F\left( g(s),\psi \right) $ a.e.

Now, each $F\left( g(s),\nu _{n}\right) $ is measurable because $F$ is
Caratheodory (recall from \cite[Appendix B]{brst} that $g$ has the property
that for any measurable $r$, if $r\circ g$ is well-defined, $r\circ g$ is
measurable). Hence $\widetilde{F}(\cdot ,\phi )$ is measurable for fixed $%
\phi \in \mathscr{C}([-r,0];E)$.

The bound on the operator norm follows from \eqref {8.08}.
\end{proof}

\begin{remark}
Assume \eqref {11.29}. Given a solution $x$ of \eqref {7.41}, its norm in $E$
satisfies a first order scalar delay equation with bounded parameters.
Namely, 
\begin{equation*}
|\left\{ \frac{d}{dt}|x(t)|\right\} |\leq \max_{s\in \lbrack t-\tau
,t]}|x(s)|.
\end{equation*}%
Thus, also 
\begin{equation*}
\frac{d}{dt}|x(t)|=p(t)\max_{s\in \lbrack t-\tau ,t]}|x(s)|,
\end{equation*}%
where 
\begin{equation*}
p(t):=\left\{ 
\begin{array}{cc}
\frac{\frac{d}{dt}|x(t)|}{\max_{s\in \lbrack t-\tau ,t]}|x(s)|}, & 
\max_{s\in \lbrack t-\tau ,t]}|x(s)|>0 \\ 
0, & \max_{s\in \lbrack t-\tau ,t]}|x(s)|=0%
\end{array}%
\right. ,|p(t)|\leq 1.
\end{equation*}%
It is trivial that 
\begin{equation*}
\max_{s\in \lbrack t-\tau ,t]}|x(s)|=|x(t-\sigma (t))|,
\end{equation*}%
for some $\sigma (t)$ measurable satisfying $0\leq \sigma (t)\leq \tau $.
\end{remark}

\bigskip

\begin{lemma}[\protect\cite{bbs,kurz1, myshkispaper,myshkis, lillo,sb}]
\label{Lemma8.12 copy(1)}Assume \eqref {11.29} and let $x$ be a solution of %
\eqref {7.41} on $[0,\tau ]$ corresponding to initial function $\phi $. Let $%
i$ be a nonnegative integer, and assume $x(0+k\frac{\tau }{2^{i-1}}%
)=0,k=0,1,2,...,2^{i-1}$. Then if $\tau \leq 2^{i}$ we have \ 
\begin{equation*}
|x(t)|\leq \frac{\tau }{2^{i}}|\phi |,t\in \lbrack 0,\tau ]
\end{equation*}%
and if $\tau \geq 2^{i}$ we have 
\begin{equation*}
|x(t)|\leq \exp \left( \tau -2^{i}\right) |\phi |,t\in \lbrack 0,\tau ].
\end{equation*}
\end{lemma}

\begin{corollary}
\bigskip \label{Corollary3.43}Assume \eqref {11.29} and that $E$ has finite
dimension $d$. Then (uniformly for any $n\in 
\mathbb{Z}
$) the operator $Q_{n}$ is compact, and there exist subspaces $X^{i}$ of $%
\mathscr{C}\left( [-\tau ,0];E\right) $, such that 
\begin{eqnarray*}
\mathscr{C}\left( [-\tau ,0];E\right) &=&X^{0}\supset ...\supset
X^{i}\supset X^{i+1}\supset ... \\
\func{codim}X^{i} &\leq &k_{i}, \\
|Q_{n}|_{X^{i}}| &\leq &\rho _{i},
\end{eqnarray*}%
where 
\begin{equation*}
k_{i}=\left\{ 
\begin{array}{cc}
0, & i=0 \\ 
d, & i=1 \\ 
2d, & i=2 \\ 
\left( 2^{i-2}+1\right) d, & i>2%
\end{array}%
\right. ,
\end{equation*}%
and either 
\begin{equation*}
\rho _{i}=\left\{ 
\begin{array}{cc}
\exp \left( \tau \right) , & i=0 \\ 
\frac{\tau }{2^{i-1}}, & i\geq 1%
\end{array}%
\right. ,\text{ for }\tau \in (0,1],
\end{equation*}%
or 
\begin{equation*}
\rho _{i}=\left\{ 
\begin{array}{cc}
\exp \left( \tau \right) , & i=0 \\ 
\exp \left( \tau -2^{i-1}\right) , & i\leq \frac{\ln \tau }{\ln 2}+1 \\ 
\frac{\tau }{2^{i-1}}, & i>\frac{\ln \tau }{\ln 2}+1%
\end{array}%
\right. ,\text{ for }\tau \geq 1.
\end{equation*}
\end{corollary}

\begin{corollary}
\bigskip Assume \eqref {11.29} and that $E$ has finite dimension $d$. Assume
that \eqref {7.41} is stable in the following sense:%
\begin{equation}
\sup_{t\geq s}|x_{t}(\phi )|\leq M|\phi |,  \label{4.33}
\end{equation}%
where $x_{t}(\phi )\in \mathscr{C}\left( [-\tau ,0];E\right) $ is the
solution on $[s,+\infty )$ corresponding to initial function $\phi $, and %
\eqref {4.33} holds uniformly in $\phi \in \mathscr{C}\left( [-\tau
,0];E\right) ,s\in \lbrack 0,+\infty )$. Then the space of initial functions
corresponding to solutions which do not tend to zero is of dimension no
greater than $k_{i}$, if $\rho _{i}M<1$.
\end{corollary}

\begin{proof}
Easy consequence of the Bolzano-Weierstrass property of finite dimensional
spaces, together with Corollary {\ref{Corollary3.43}.}
\end{proof}

\begin{remark}
Similar bounds may be obtained for any $n-$adic subdivision of the delay
interval, as for the dyadic subdivision above.
\end{remark}

\subsection{Variational equation}

We now proceed to the nonlinear case, where $F:%
\mathbb{R}
\times \mathscr{C}\left( [-\tau ,0];E\right) \rightarrow E$ again satisfies
the Caratheodory conditions, as in the previous section. The following proof
is partially inspired by \cite{hart,mik}.

\begin{theorem}
Assume $F(t,\cdot ):\mathscr{C}\left( [-\tau ,0];E\right) \circlearrowleft $
has a directional (Gateaux) derivative for every direction at every $x\in %
\mathscr{C}\left( [-\tau ,0];E\right) $, denoted $d_{2}F(t,x)[\cdot ]$. We
further assume that for each $x\in \mathscr{C}\left( [-\tau ,+\infty
);E\right) $ we have that $d_{2}F(t,x_{t})[\cdot ]$ has a norm locally
Lebesgue measurable wrt to time $t$, and such that $d_{2}F(\cdot ,x_{t})[y]$
is Lebesgue measurable for each $y\in \mathscr{C}\left( [-\tau ,0];E\right) $%
. Fix a solution $x_{t}(\phi )$ of \eqref {7.41} on $[0,+\infty )$ and set $%
V(t;\phi ,\xi )$ to be the solution of 
\begin{eqnarray*}
V^{\prime }(t) &=&\left[ d_{2}F(t,x_{t})\right] V_{t},t\geq 0 \\
V_{0} &=&\xi \in \mathscr{C}\left( [-\tau ,0];E\right) .
\end{eqnarray*}%
Then for any $\phi \in \mathscr{C}\left( [-\tau ,0];E\right) $, any $T>0$,
and $\xi \in \mathscr{C}\left( [-\tau ,0];E\right) $ we have 
\begin{equation*}
\lim_{|\xi |\rightarrow 0}\frac{1}{|\xi |}|x_{t}(\phi +\xi )-x_{t}(\phi
)-V_{t}(\phi ,\xi )|=0,
\end{equation*}%
uniformly for $t\in \lbrack 0,T]$.
\end{theorem}

\begin{proof}
\bigskip We have for fixed $T>0$, 
\begin{eqnarray*}
x(t;\phi +\xi ) &=&\phi (0)+\xi (0)+\int_{0}^{t}d_{2}F(s,x_{s}(\phi +\xi
))ds, \\
V(t;\phi ,\xi ) &=&\xi (0)+\int_{0}^{t}\left[ d_{2}F(s,x_{s}(\phi ))\right]
V_{s}ds.
\end{eqnarray*}%
Set%
\begin{equation*}
\omega (t;x,y)=F(t,y)-F(t,x)-\left[ d_{2}F(t,x)\right] (y-x),t>0,x,y\in %
\mathscr{C}\left( [-\tau ,0];E\right) .
\end{equation*}%
Notice that by assumption we have (pointwise in time) 
\begin{equation}
\lim_{y\rightarrow x}\frac{|\omega (t;x,y)|}{|x-y|}=0.  \label{6.17}
\end{equation}%
We also set%
\begin{eqnarray*}
\delta (t) &=&x(t;\phi +\xi )-x(t;\phi )-V(t;\phi ,\xi ), \\
\delta _{t} &=&x_{t}(\phi +\xi )-x_{t}(\phi )-V_{t}(\phi ,\xi ), \\
\Delta (t) &=&x(t;\phi +\xi )-x(t;\phi ), \\
\Delta _{t} &=&x_{t}(\phi +\xi )-x_{t}(\phi ).
\end{eqnarray*}%
With the above notation we have%
\begin{eqnarray*}
\delta (t) &=&\int_{0}^{t}\left\{ F(s,x_{s}(\phi +\xi ))-F(s,x_{s}(\phi ))- 
\left[ d_{2}F(s,x_{s}(\phi )\right] V_{s}\right\} ds \\
&=&\int_{0}^{t}\left\{ \omega (s;x(t;\phi +\xi ),x(t;\phi ))-\left[
d_{2}F(s,x_{s}(\phi )\right] \delta _{t}\right\} ds.
\end{eqnarray*}%
The triangle inequality and \eqref {8.08} then give%
\begin{equation}
|\delta (t)|\leq \int_{0}^{t}|\omega (s;x(s;\phi +\xi ),x(s;\phi
))|ds+\int_{0}^{t}n(s)|\delta _{s}|ds.  \label{6.21}
\end{equation}%
Furthermore, we have by definition 
\begin{eqnarray*}
&&\frac{|\omega (s;x(s;\phi +\xi ),x(s;\phi ))|}{|\xi |} \\
&\leq &\frac{|F(s,x(s;\phi +\xi ))-F(s,x(s;\phi ))|}{|\xi |}+\frac{%
|d_{2}F(t,x(s;\phi ))||x(s;\phi +\xi )-x(s;\phi )|}{|\xi |}.
\end{eqnarray*}%
Applying \eqref {8.08}, \eqref {6.13}, the last inequality yields%
\begin{equation}
\frac{|\omega (s;x(s;\phi +\xi ),x(s;\phi ))|}{|\xi |}\leq 2c(T)n(s).
\label{6.16}
\end{equation}%
Now \eqref {8.08}, \eqref {6.17}, also imply%
\begin{eqnarray*}
\frac{|\omega (s;x(s;\phi +\xi ),x(s;\phi ))|}{|\xi |} &=&\frac{|\omega
(s;x(s;\phi +\xi ),x(s;\phi ))|}{|x(s;\phi +\xi )-x(s;\phi )|}\frac{%
|x(s;\phi +\xi )-x(s;\phi )|}{|\xi |} \\
&\leq &\frac{|\omega (s;x(s;\phi +\xi ),x(s;\phi ))|}{|x(s;\phi +\xi
)-x(s;\phi )|}c(T),
\end{eqnarray*}%
and%
\begin{equation}
\lim_{|\xi |\rightarrow 0}\frac{|\omega (s;x(s;\phi +\xi ),x(s;\phi ))|}{%
|\xi |}=0.  \label{6.18}
\end{equation}%
Now we apply the dominated convergence Theorem using \eqref {6.16}, %
\eqref{6.18}, obtaining 
\begin{equation}
\lim_{|\xi |\rightarrow 0}\frac{\int_{0}^{t}|\omega (s;x(s;\phi +\xi
),x(s;\phi ))|ds}{|\xi |}=0.  \label{6.20}
\end{equation}%
We rewrite \eqref {6.21} as%
\begin{equation*}
\frac{|\delta (t)|}{|\xi |}\leq \frac{\int_{0}^{t}|\omega (s;x(s;\phi +\xi
),x(s;\phi ))|ds}{|\xi |}+\int_{0}^{t}n(s)\frac{|\delta _{s}|}{|\xi |}ds,
\end{equation*}%
which, together with \eqref {6.20}, immediately gives the desired result.
\end{proof}

\section{Acknowledgements}

Many thanks are due to Alex Blumenthal and Mih\'{a}ly Pituk, for deep
discussions on the notion of \textquotedblleft determinant\textquotedblright
, \ to Ferenc Hartung for insights on the optimal requirements for the
variational equation, and to Matt Baker for advice on tackling certain
nonarchimedean challenges.

\end{document}